\documentclass[a4paper,10pt]{amsart}

\usepackage[bookmarks]{hyperref}
\usepackage{amsfonts}
\usepackage{amsmath}
\usepackage{amssymb,amsxtra,amscd,amsthm}
\usepackage{mathrsfs}
\usepackage{mathtools}
\usepackage{todonotes}
\usepackage[foot]{amsaddr}

\newcommand{\R}{\mathbb{R}}
\newcommand{\C}{\mathbb{C}}

\newcommand{\N}{\mathbb{N}}

\newcommand{\supp}{\mbox{supp}\,}
\newcommand{\dist}{\mbox{dist}\,}

\newtheorem{theorem}{Theorem}[section]
\newtheorem{lemma}[theorem]{Lemma}
\newtheorem{corollary}[theorem]{Corollary}
\newtheorem{proposition}[theorem]{Proposition}

\theoremstyle{definition}
\newtheorem{remark}[theorem]{Remark}
\newtheorem{definition}[theorem]{Definition}
\newtheorem{example}[theorem]{Example}

\allowdisplaybreaks

\title[Topologizable and power bounded weighted composition operators]{Topologizable and power bounded weighted composition operators on spaces of distributions}

\author{Thomas Kalmes}
\address{Chemnitz University of Technology, Faculty of Mathematics, 09107 Chemnitz, Germany}
\email{thomas.kalmes@math.tu-chemnitz.de}

\begin{document}

\begin{abstract}
	We study topologizability and power boundedness of weigh\-ted composition operators on (certain subspaces of) $\mathscr{D}'(X)$ for an open subset $X$ of $\mathbb{R}^d$. For the former property we derive a characterization in terms of the symbol and the weight of the weighted composition operator, while for the latter property necessary and sufficient conditions on the weight and the symbol are presented. Moreover, for an unweighted composition operator a characterization of power boundedness in terms of the symbol is derived for the special case of a bijective symbol.\\
	
	
\end{abstract}

\maketitle

\section{Introduction}

Recently, topologizability and power boundedness (see Definition \ref{def topologizability etc} below) of (weighted) composition operators on various spaces of functions have been studied by several authors, see e.g.\ \cite{BeGoJoJo16-2}, \cite{BoDo11-2}, \cite{GoJoJo16} \cite{Wolf12-1}, \cite{Wolf12-2}, \cite{Wolf15}. In \cite{Kalmes19-2}, a general approach within the framework of function spaces defined by local properties which are subspaces of continuous functions on a locally compact, $\sigma$-compact, non-compact Hausdorff space has been provided. By this general framework, many function spaces which appear in mathematical analysis are covered, and topologizability and power boundedness of weighted composition operators on such spaces are characterized in terms of the symbol and the weight of the operator. However, this general setting does not contain the space of distributions over an open subset of $\R^d$.

The objective of the present note is to characterize topologizability of weighted composition operators on spaces of distributions defined by local properties. Moreover, we investigate power boundedness in this setting as well, and characterize this property for unweighted composition operators on $\mathscr{D}'(X), X\subseteq\R^d$ open, in terms of the symbol for the special case of a bijective symbol.

While the interest for power boundedness of an operator stems from its close relationship to (uniform) mean ergodicity, topologizable operators were introduced by \.{Z}elazko in \cite{Zelazko07} (see also \cite{Bonet07}). For a Hausdorff locally convex space $E$, in order that the algebra $L(E)$ of all continuous endomorphisms of $E$ (with composition as multiplication) is topologizable, i.e.\ $L(E)$ admits a locally convex topology for which multiplication is jointly continuous, $E$ is necessarily subnormed. The latter means that there is a norm on $E$ such that the corresponding topology is finer than the locally convex topology initially given on $E$, see \cite{Zelazko02} and references therein. In case of a sequentially complete $E$ it has been shown in \cite{Zelazko02} that this necessary condition on $E$ is also sufficient for the topologizability of $L(E)$. Motivated by this, in \cite{Zelazko07} it was investigated when for a given continuous linear operator $T$ on a Hausdorff locally convex space $E$ there is a unital subalgebra $A$ of $L(E)$ which contains $T$ and which admits a locally convex topology making $A$ into a topological algebra such that additionally the map
$$A\times E\rightarrow E, (S,x)\mapsto Sx$$
is continuous. By \cite[Theorem 5]{Zelazko07} for a given $T\in L(E)$ there is such a subalgebra $A$ of $L(E)$ precisely when $T$ is topologizable.\\

Throughout, we use standard notation and terminology from functional analysis. For anything related to functional analysis which is not explained in the text we refer the reader to \cite{MeVo1997}. Moreover, we use common notation from the theory of distributions and linear partial differential operators. For this we refer the reader to \cite{HoermanderPDO1} and \cite{HoermanderPDO2}.

By an open, relatively compact exhaustion $(X_n)_{n\in\N}$ of a topological space $X$ we understand a sequence of open subsets of $X$ such that $\overline{X}_n\subseteq X_{n+1}$ with compact closure $\overline{X}_n$ for all $n\in\N$ such that $\cup_{n\in\N} X_n=X$.

\section{Weighted composition operators on spaces of distributions defined by local properties}\label{distributions defined by local properties}

As in \cite{Kalmes19-2} we are interested in topologizability of weighted composition operators - the precise definition of topologizability will be recalled below. However, in contrast to \cite{Kalmes19-2} where weighted composition operators were considered on spaces of functions, in the present paper we consider these operators on spaces of distributions defined by local properties. As a general framework we choose the notion of sheaves. In what follows we always assume that the space of compactly supported smooth functions $\mathscr{D}(X)$ on an open set $X\subseteq\R^d$ is equipped with its standard locally convex topology (see e.g.\ \cite[Chapter 6]{RudinFA} or \cite[Chapter 2.12]{Horvath}).

\begin{definition}\label{sheaf of distributions}
	From now on we assume that $\mathscr{G}$ is a \textit{sheaf of distributions on $\R^d$ defined by local properties}, i.e.\ we assume that the following properties hold.
	\begin{itemize}
		\item For every open subset $X\subseteq\R^d$, $\mathscr{G}(X)$ is a subspace of $\mathscr{D}'(X)$ equipped with the subspace topology. Here, as usual $\mathscr{D}'(X)$ is equipped with the strong (dual) topology on $\mathscr{D}'(X)$ with respect to the dual pair $(\mathscr{D}(X),\mathscr{D}'(X))$. Moreover, whenever $Y\subseteq\R^d$ is another open set with $Y\subseteq X$ we assume that the restriction mapping
		$$r_X^Y:\mathscr{G}(X)\rightarrow\mathscr{G}(Y),u\mapsto u_{|Y}$$
		is well defined. Here we use the common abbreviation $u_{|Y}:=u_{|\mathscr{D}(Y)}$ for $u\in\mathscr{D}'(X)$.
		
		\item (Localization) For an open set $X\subseteq\R^d$, for every open cover $(X_\iota)_{\iota\in I}$ of $X$, and for each $u,v\in\mathscr{G}(X)$ with $u_{|X_\iota}=v_{|X_\iota} (\iota\in I)$ we have $u=v$. (Note that this property always holds since $\mathscr{D}'$ is a sheaf!)
		
		\item (Gluing) For an open set $X\subseteq\R^d$, for every open cover $(X_\iota)_{\iota\in I}$ of $X$, and for all $(u_\iota)_\iota\in\prod_{\iota\in I}\mathscr{G}(X_\iota)$ with $u_{\iota|X_\iota\cap X_\kappa}=u_{\kappa|X_\iota\cap X_\kappa}\, (\iota,\kappa\in I)$ there is $u\in\mathscr{G}(X)$ with $u_{|X_\iota}=u_\iota\,(\iota\in I)$.
	\end{itemize}
	It follows from the above properties that for every open subset $X\subseteq\R^d$ and each open, relatively compact exhaustion $(X_n)_{n\in\N_0}$ of $X$ the space $\mathscr{G}(X)$ and the projective limit $\mbox{proj}_{\leftarrow n}(\mathscr{G}(X_n),r_{X_{n+1}}^{X_n})$ are algebraically isomorphic via the mapping
	$$\mathscr{G}(X)\rightarrow\mbox{proj}_{\leftarrow n}(\mathscr{G}(X_n),r_{X_{n+1}}^{X_n}),u\mapsto (r_X^{X_n}(u))_{n\in\N_0}=(u_{|X_n})_{n\in\N_0}.$$
	For a bounded subset $B\subseteq\mathscr{D}(X)$ it follows (see e.g.\ \cite[Theorem 6.5]{RudinFA} or \cite[Example 2.12.6]{Horvath}) that there is $n\in\N$ for which $B\subseteq\mathscr{D}(X_n)$ and that $B$ is bounded in $\mathscr{D}(X_n)$. From this we conclude that the above algebraic isomorphism between $\mathscr{G}(X)$ and $\mbox{proj}_{\leftarrow n}(\mathscr{G}(X_n),r_{X_{n+1}}^{X_n})$ is a topological isomorphism.
	
	For obvious reasons, $\mathscr{G}(X)$, $X\subseteq\R^d$ open, is called a \textit{space of distributions defined by local properties}.
\end{definition}

\begin{example}
	Obviously, we can choose $\mathscr{G}=\mathscr{D}'$. Moreover, for any polynomial $P\in\C[X_1,\ldots,X_d]$ we may consider $\mathscr{G}=\mathscr{D}'_P$, i.e.\ for every open $X\subseteq\R^d$
	$$\mathscr{D}'_P(X)=\left\{u\in\mathscr{D}'(X);\,P(\partial)u=0\right\}.$$
	Whenever $P$ is not hypoelliptic, $\mathscr{D}'_P(X)$ and $C_P^\infty(X)$ do not coincide, where $C_P^\infty(X)=\left\{f\in C^\infty(X);\,P(\partial)f=0\right\}$. Weighted composition operators on $C_P^\infty(X)$ have been studied in \cite[Section 4]{Kalmes19-2} and \cite[Section 6]{Kalmes19-3}.
\end{example}

\begin{definition}\label{def topologizability etc}
	For a locally convex space $E$ we denote by $cs(E)$ the set of continuous seminorms on $E$. Let $T$ be a continuous linear operator on $E$.
	\begin{enumerate}
		\item[i)] $T$ is said to be \textit{topologizable} if for every $p\in cs(E)$ there is $q\in cs(E)$ such that for all $m\in\N$ there is $\gamma_m>0$ with
		$$p(T^m(x))\leq \gamma_m q(x)\,\mbox{ for all }x\in E.$$
		\item[ii)] $T$ is said to be \textit{power bounded}
		if for every $p\in cs(E)$ there is $q\in cs(E)$ such that for all $m\in\N$
		$$p(T^m(x))\leq q(x)\,\mbox{ for all }x\in E,$$
		i.e.\ if the set of iterates $\left\{T^m;\,m\in\N_0\right\}$ of $T$ is equicontinuous.
	\end{enumerate}
\end{definition}

Clearly, every power bounded operator is topologizable. Moreover, $T$ is topologizable whenever there is a sequence $(\alpha_m)_{m\in\N}$ of strictly positive numbers such that the set $\left\{\alpha_m T^m;\,m\in\N\right\}$ is equicontinuous and then the sequences $(\gamma_m)_{m\in\N}$ in the definition of topologizability can be chosen independently of the seminorms involved $p$ and $q$.

\begin{definition}\label{weighted composition operator}
	Let $X\subseteq\R^d$ be open, $w\in C^\infty(X)$, and let $\psi:X\rightarrow X$ be smooth such that $\det J\psi(x)\neq 0$ for every $x\in X$, where $J\psi(x)$ denotes the Jacobian. The \textit{weighted composition operator} $C_{w,\psi}$ on $\mathscr{D}'(X)$ is defined as the unique continuous operator on $\mathscr{D}'(X)$ which extends $C_{w,\psi}$ on $C(X)$ defined as $C_{w,\psi}(f)=w\cdot\,(f\circ\psi)$, see e.g.\ \cite[Theorem 6.1.2]{HoermanderPDO1}. The function $\psi$ is called the \textit{symbol} and $w$ the \textit{weight} of $C_{w,\psi}$. For the special case $w=1$ we write $C_\psi$ instead of $C_{1,\psi}$ and $C_\psi$ is simply called composition operator.
\end{definition}

If $C_{w,\psi}(\mathscr{G}(X))\subseteq\mathscr{G}(X)$ it follows that $C_{w,\psi}$ is a continuous operator on $\mathscr{G}(X)$. We are interested to characterize when $C_{w,\psi}$ is topologizable etc.\ on $\mathscr{G}(X)$. For injective $\psi$ one verifies
\begin{align*}
\langle C_{w,\psi}^m(u),\varphi\rangle&=\left\langle u,\left(\varphi \frac{\prod_{j=0}^{m-1} w(\psi^j(\cdot))}{|\det J\psi^m(\cdot)|}\right)\circ(\psi^m)^{-1}\right\rangle\\
&=\left\langle u,\left(\varphi \prod_{j=0}^{m-1} \frac{w(\psi^j(\cdot))}{|\det J\psi(\psi^j(\cdot))|}\right)\circ(\psi^m)^{-1}\right\rangle
\end{align*}
for all $u\in\mathscr{D}'(X),\varphi\in\mathscr{D}(X), m\in\N$. 

\begin{definition}
	Let $X$ be a topological space and $\psi: X\rightarrow X$ be a continuous mapping. $\psi$ is said to have \textit{stable orbits} if for every compact subset $K\subseteq X$ there is another compact subset $L\subseteq X$ with $\psi^m(K)\subseteq L$ for every $m\in\N$.
\end{definition}

Our first result gives a sufficient condition on the symbol $\psi$ for the weighted composition operator $C_{w,\psi}$ to be topologizable. Clearly, for every topologizable operator $T$ on a locally convex space $E$ and every $T$-invariant subspace $F$ of $E$ the restriction of $T$ to $F$ is again topologizable. Hence, in the situation of the proposition below, given a sheaf $\mathscr{G}$ of distributions on $\R^d$ such that $C_{w,\psi}(\mathscr{G}(X))\subseteq\mathscr{G}(X)$ the restriction of $C_{w,\psi}$ to $\mathscr{G}(X)$ is topologizable if $\psi$ has stable orbits.

\begin{proposition}\label{stable sufficient for distributions}
	Let $X\subseteq\R^d$ be open, $w\in C^\infty(X)$, $\psi:X\rightarrow X$ be smooth and injective such that $\det J\psi(x)\neq 0$ for all $x\in X$. Assume that $\psi$ has stable orbits. Then $C_{w,\psi}$ is topologizable on $\mathscr{D}'(X)$. 	
\end{proposition}

\begin{proof}
	Let $K\subseteq X$ be compact and choose $L\subseteq X$ compact such that $\psi^m(K)\subseteq L$ for all $m\in\N$. For each $m\in\N$
	$$M_m:\mathscr{D}(K)\rightarrow\mathscr{D}(K),\varphi\mapsto\varphi \prod_{j=0}^{m-1}\frac{w(\psi^j(\cdot))}{|\det J\psi(\psi^j(\cdot))|}$$
	is continuous as is
	$$\Psi_m:\mathscr{D}(K)\rightarrow\mathscr{D}(\psi^m(K)),\varphi\mapsto\varphi\circ(\psi^m)^{-1}.$$
	Therefore, for every absolutely convex and bounded subset $B\subseteq\mathscr{D}(K)$ we obtain together with the correctly defined continuous inclusions
	$$\forall\,m\in\N:\mathscr{D}(\psi^m(K))\hookrightarrow\mathscr{D}(L)$$
	that
	$$\forall\,m\in\N: B_m:=\left(\Psi_m\circ M_m\right)(B)\subseteq \mathscr{D}(L)\mbox{ absolutely convex and bounded}.$$
	Since $\mathscr{D}(L)$ is metrizable it follows from Mackey's countability condition (see e.g.\ \cite[Proposition 2.6.3]{Horvath} or \cite[Lemma 26.6 a)]{MeVo1997}) that there are $\tilde{B}\subseteq\mathscr{D}(L)$ bounded, absolutely convex and closed, $(\alpha_m)_{m\in\N}$ in $(0,\infty)$ such that
	$$\forall\,m\in\N:\,B_m\subseteq\alpha_m \tilde{B}.$$
	Thus, we obtain for $\varphi\in B, u\in\mathscr{D}'(X)$
	\begin{align*}
	|\langle C_{w,\psi}^m(u),\varphi\rangle|&=|\langle u, (\Psi_m\circ M_m)\varphi\rangle|=\alpha_m\left|\left\langle u,\frac{1}{\alpha_m}\left(\Psi_m\circ M_m\right)\varphi\right\rangle\right|\\
	&\leq\alpha_m\sup\left\{|\langle u,\phi\rangle|;\,\phi\in \tilde{B}\right\}
	\end{align*}
	which implies for all $m\in\N$ and $u\in\mathscr{D}'(X)$
	$$\sup\left\{|\langle C_{w,\psi}^m (u),\varphi\rangle|;\,\varphi\in B\right\}\leq\alpha_m\sup\left\{|\langle u,\phi\rangle|;\,\phi\in \tilde{B}\right\}.$$
	Because every absolutely convex and bounded subset $B\subseteq\mathscr{D}(X)$ is contained in $\mathscr{D}(K)$ and bounded in $\mathscr{D}(K)$ for a suitable compact set $K\subseteq X$ (see e.g.\ \cite[Theorem 6.5]{RudinFA} or \cite[Example 2.12.6]{Horvath}) the proof is finished.
\end{proof}

The next result shows that under suitable additional hypothesis on $\mathscr{G}(X)$ as well as on $w$ and $\psi$, topologizability of $C_{w,\psi}$ implies that $\psi$ has stable orbits.  

\begin{lemma}\label{stable necessary distributions}
	Let $\mathscr{G}$ be a sheaf of distributions defined by local properties, $X\subseteq\R^d$ be open, $w\in C^\infty(X)$, $\psi:X\rightarrow X$ be smooth and injective such that $\det J\psi(x)\neq 0$ for all $x\in X$. Assume that $C_{w,\psi}(\mathscr{G}(X))\subseteq\mathscr{G}(X)$ and that additionally the following conditions hold.
	\begin{enumerate}
		\item[a)] There is an open, relatively compact exhaustion $(X_n)_{n\in\N}$ of $X$ such that for each $n\in\N$, every $x\in X\backslash\overline{X}_n$, and every $\varepsilon>0$ for which $\overline{B(x,\varepsilon)}\subseteq X\backslash{X}_n$, the restriction
		$$r_X^{X_n\cup B(x,\varepsilon)}:\mathscr{G}(X)\rightarrow\mathscr{G}(X_n\cup B(x,\varepsilon))$$
		has dense range, where $B(x,\varepsilon)$ denotes the open euclidean ball around $x$ with radius $\varepsilon$.
		\item[b)] There is $\varepsilon_0>0$ such that for all $\varepsilon\in (0,\varepsilon_0)$ there is $\chi_\varepsilon\in\mathscr{D}(B(0,\varepsilon))$ such that for all $x\in X$ with $\overline{B(x,\varepsilon)}\subseteq X$ there is $h\in\mathscr{G}(B(x,\varepsilon))$ satisfying
		$$\langle h,\tau_x\chi_\varepsilon\rangle\neq 0,$$
		where $\tau_x\chi_\varepsilon(y):=\chi_\varepsilon(y-x)$.
		\item[c)] For every $l\in\N_0$ the set $\left\{x\in X;\, w(\psi^l(x))\neq0\right\}$ is dense in $X$.
	\end{enumerate}
	If $C_{w,\psi}$ is topologizable on $\mathscr{G}(X)$, then $\psi$ has stable orbits.
\end{lemma}

\begin{remark}\label{sufficient a) and b) distributions}
	Before we present the technical proof of the above lemma we take a closer look at its additional assumptions a) - c).
	
	Under the hypothesis on $\psi$ in the above lemma it follows that for every $x\in X$ there is an open neighborhood $U_x\subseteq X$ such that $\psi_{|U_x}$ is (injective and) open. Hence, if for $w\in C^\infty(X)$ the set $\left\{x\in X;\, w(x)\neq 0\right\}$ is dense in $X$, it follows from \cite[Proposition 3.9]{Kalmes19-2} that the above hypothesis c) is fulfilled. 
	
	The above hypothesis b) is satisfied whenever $\mathscr{G}$ is invariant under translations (i.e.\ whenever for $u\in\mathscr{G}(X)$ we have $\tau_x u\in\mathscr{G}(-x+X)$, where
	$$\forall\,\varphi\in\mathscr{D}(-x+X):\,\langle\tau_x u,\varphi\rangle=\langle u, \tau_{-x}\varphi\rangle)$$
	and $\mathscr{G}$ satisfies
	$$\exists\varepsilon_0>0\forall\varepsilon\in (0,\varepsilon_0)\exists\chi\in\mathscr{D}\left(B(0,\varepsilon)\right), h\in\mathscr{G}\left(B(0,\varepsilon)\right):\,\langle h,\chi\rangle\neq 0.$$
	Apart from $\mathscr{G}=\mathscr{D}'$ this is in particular the case for $\mathscr{G}=\mathscr{D}'_P$ whenever $P\in\C[X_1,\ldots,X_d], d\geq 2,$ is non-constant. Indeed, while translation invariance is obvious, since $P$ is non-constant
	$$V(P)=\left\{\zeta\in\C^d;\,P(\zeta)=0\right\}$$
	is neither empty nor discrete. With $e_\zeta(x)=\exp\left(\sum_{j=1}^d x_j\zeta_j\right),\zeta\in\C^d,$ we have $e_\zeta\in\mathscr{D}'_P(X),\zeta\in V(P),$ and for every $\chi\in\mathscr{D}(\R^d)$ the Fourier-Laplace Transform
	$$\C^d\rightarrow\C,\zeta\mapsto\hat{\chi}(\zeta)=\langle e_\zeta,\chi\rangle$$
	is analytic so that $\hat{\chi}_{|V(P)}=0$ implies $\hat{\chi}=0$ and thus $\chi=0$.
	
	Finally, as far as hypothesis a) of the above lemma is concerned, it is satisfied for $\mathscr{G}=\mathscr{D}'$. Indeed, using that every bounded subset of $\mathscr{D}(X)$ is contained and bounded in $\mathscr{D}(K)$ for suitable compact $K\subseteq X$ and using multiplication with compactly supported $\varphi\in\mathscr{D}(X)$ for which $\varphi=1$ in a neighborhood of $K$ one easily verifies that $\mathscr{E}'(\R^d)$ is dense in $\mathscr{D}'(X)$. Therefore, $\mathscr{G}=\mathscr{D}'$ fulfils hypothesis a) for each $X$ and an arbitrary open, relatively compact exhaustion $(X_n)_{n\in\N}$. 
	
	Moreover, for $X\subseteq\R^d, d\geq 3,$ let
	$$\forall\,n\in\N:\,X_n:=\left\{x\in X;\,|x|<n, \dist(x,\R^d\backslash X)>\frac{1}{n}\right\},$$
	so that $(X_n)_{n\in\N}$ is a relatively compact exhaustion of $X$. Given a polynomial $P\in\C[X_1,\ldots,X_d]$ with principal part $P_m$ such that $\left\{\xi\in\R^d;\,P_m(\xi)=0\right\}$ is contained in a one-dimensional subspace of $\R^d$ and such that $P(\partial)$ is surjective on $C^\infty(X)$, it follows from \cite[Theorem 4.4 ii)]{Kalmes19-2} combined with \cite[Theorem 3.1]{Kalmes18-2} that hypothesis a) of the above corollary is fulfilled for $\mathscr{G}=\mathscr{D}'_P$. Of course, for (hypo)elliptic $P$ we have $C_P^\infty=\mathscr{D}'_P$ (topologically!) and this sheaf is covered by \cite[Section 4]{Kalmes19-2}. However, for the (non-hypoelliptic) time-dependent free Schr\"odinger operator $P(\partial)=i\frac{\partial}{\partial t}+\Delta_x$ it follows that hypothesis a) (and b)) of Lemma \ref{stable necessary distributions} are satisfied for the corresponding sheaf $\mathscr{D}'_P$ provided that $P(\partial)$ is surjective on $C^\infty(X)$. A geometric/topological characterization of those $X$ fulfilling the latter property was recently given in \cite[Corollary 5]{Kalmes19-1}.
\end{remark}

\textit{Proof of Lemma \ref{stable necessary distributions}.}
Let $(X_n)_{n\in\N}$ be the open, relatively compact exhaustion of $X$ from hypothesis a). Clearly, the claim will follow once we have shown that for every $n\in\N$ there is $k\in\N$ with $\psi^m(\overline{X}_n)\subseteq \overline{X}_k$ for all $m\in\N$. In order to do so, some technical preparations have to be made which will be finished once we have proved (\ref{special test function}) below.

For $n\in\N$ and $m\in\N_0$ we define
$$\delta_{m,n}:=\dist\left(\psi^m(\overline{X}_n),\R^d\backslash\psi^m(X_{n+1})\right)$$
and
$$\delta_n:=\delta_{0,n},$$
so that $\delta_{m,n}>0$. It follows from hypothesis c) that the set
$$\bigcap_{l=0}^{m-1}\left\{x\in X;\,w(\psi^l(x))\neq 0\right\}$$
is dense in $X$ for every $m\in\N$. For $n,m,l\in\N$ with $l>\frac{2}{\delta_n}$ it follows that $B\left(x,2/l\right)\subseteq X_{n+1}$ whenever $x\in\overline{X}_n$ and we define
$$Y_{l,m,n}:=\left\{x\in\overline{X}_n;\,\forall\,y\in B\left(x,\frac{2}{l}\right):\,\prod_{j=0}^{m-1}w(\psi^j(y))\neq 0\right\}$$
and
$$X_{l,m,n}:=\bigcup_{x\in Y_{l,m,n}}B\left(x,\frac{2}{l}\right)$$
so that $X_{l,m,n}\subseteq X_{n+1}$. We then have
\begin{equation}\label{inclusion 1}
\overline{\bigcup_{x\in Y_{l,m,n}}B\left(x,\frac{1}{l}\right)}\subseteq X_{l,m,n}.
\end{equation}
Indeed, if $y\in\overline{\cup_{x\in Y_{l,m,n}}B\left(x,1/l\right)}$ there are sequences $(x_k)_{k\in\N}$ in $Y_{l,m,n}$ and $(z_k)_{k\in\N}$ in $B\left(0,1/l\right)$ such that $(x_k+z_k)_{k\in\N}$ converges to $y$. Since $Y_{l,m,n}\subseteq\overline{X}_n$ we can assume without loss of generality that $(x_k)_{k\in\N}$ converges in $\overline{X}_n$ and $(z_k)_{k\in\N}$ converges in $\overline{B\left(0,1/l\right)}$; we denote the limits by $x_0$ and $z_0$, respectively. For $v\in B\left(x_0,2/l\right)$ and $k$ sufficiently large we have
$$|x_k-x_0|<\frac{2}{l}-|x_0-v|$$
so that
$$|x_k-v|\leq |x_k-x_0|+|x_0-v|<\frac{2}{l},$$
i.e.\ $v\in B\left(x_k,2/l\right)$ hence $\prod_{j=0}^{m-1}w(\psi^j(v))\neq 0$ because $x_k\in Y_{l,m,n}$. As $v\in B\left(x_0,2/l\right)$ was chosen arbitrarily, it follows that $x_0\in Y_{l,m,n}$ so that
$$y=x_0+z_0\in\bigcup_{x\in Y_{l,m,n}}B\left(x,\frac{2}{l}\right)=X_{l,m,n}$$
showing (\ref{inclusion 1}).

Since the bijection
$$\left(\psi^m_{|\psi^m(\overline{X}_{n+1})}\right)^{-1}:\psi^m(\overline{X}_{n+1})\rightarrow\overline{X}_{n+1}$$
is uniformly continuous, for $l>2/\delta_n$ there is $\beta_{l,m}>0$ such that
\begin{equation}\label{uniform continuity}
\forall\,\psi^m(x),\psi^m(y)\in\psi^m(\overline{X}_{n+1}), |\psi^m(x)-\psi^m(y)|<\beta_{l,m}:\,|x-y|<\frac{1}{l}.
\end{equation}
For every $l,m,n\in\N$ with $l>2/\delta_n$ we choose (with $\varepsilon_0$ from hypothesis b))
$$\varepsilon_{l,m,n}\in \left(0,\min\left\{\varepsilon_0,\frac{1}{l},\frac{\delta_{m,n}}{2},\beta_{l,m}\right\}\right)$$
and let $\chi_{l,m,n}:=\chi_{\varepsilon_{l,m,n}}\in\mathscr{D}\left(B\left(0,\varepsilon_{l,m,n}\right)\right)$ according to hypothesis b).

For every $x_0\in Y_{l,m,n}$ we have
\begin{equation}\label{special test function}
\R^d\rightarrow\C, y\mapsto\chi_{l,m,n}\left(\psi^m(y)-\psi^m(x_0)\right)\in\mathscr{D}\left(\overline{\cup_{x\in Y_{l,m,n}}B(x,1/l)}\right).
\end{equation}
Indeed, if $y\in\R^d$ is such that $\chi_{l,m,n}\left(\psi^m(y)-\psi^m(x_0)\right)\neq 0$ it follows that
$$\psi^m(y)\in B\left(\psi^m(x_0),\varepsilon_{l,m,n}\right)\subseteq\psi^m\left(\overline{X}_n\right)+B(0,\varepsilon_{l,m,n})\subseteq\psi^m(X_{n+1}),$$
where we have used $\varepsilon_{l,m,n}<\delta_{m,n}/2$ and the definition of $\delta_{m,n}$ in the last inclusion. Because $\psi^m$ in injective, we conclude $y\in X_{n+1}$. Because, moreover
$$|\psi^m(y)-\psi^m(x_0)|<\varepsilon_{l,m,n}\leq \beta_{l,m}$$
we also have $|y-x_0|<1/l$ by (\ref{uniform continuity}). Hence
$$y\in B\left(x_0,\frac{1}{l}\right)\subseteq\bigcup_{x\in Y_{l,m,n}}B\left(x,\frac{1}{l}\right)$$
so that the support of $y\mapsto\chi_{l,m,n}(\psi^m(y)-\psi^m(x_0))$ is contained in the closure of $\cup_{x\in Y_{l,m,n}}B\left(x,1/l\right)$ which proves (\ref{special test function}).

We now fix $n\in\N$. Recall that our objective is to prove the existence of $k\in\N$ satitsfying $\psi^m(\overline{X}_n)\subseteq\overline{X}_k$ for all $m\in\N$. Since for $m,l\in\N$ with $l>2/\delta_n$ we have $X_{l,m,n}\subseteq X_{n+1}$, it follows from $(\ref{inclusion 1})$ and the relative compactness of $X_{n+1}$ that the closure of $\cup_{x\in Y_{l,m,n}}B(x,1/l)$ is a compact subset of $X_{l,m,n}$. Moreover, from the definition of $Y_{l,m,n}$ it follows that $\psi^m\left(Y_{l,m,n}\right)\subseteq\psi^m\left(\overline{X}_n\right)$ so that compactness of $\psi^m\left(\overline{X}_n\right)$ implies that
$$\left\{\chi_{l,m,n}\left(\psi^m(\cdot)-\psi^m(x_0)\right);\,x_0\in Y_{l,m,n}\right\}$$
is a bounded subset of $\mathscr{D}(X_{l,m,n})$. From the definition of $X_{l,m,n}$ it follows that
$$\mathscr{D}(X_{l,m,n})\rightarrow\mathscr{D}(X_{l,m,n}),\varphi\mapsto\frac{|\det J\psi^m(\cdot)|}{\prod_{j=0}^{m-1} w(\psi^j(\cdot))}\varphi$$
is well defined and continuous so that
$$\left\{\frac{|\det J\psi^m(\cdot)|}{\prod_{j=0}^{m-1} w(\psi^j(\cdot))}\chi_{l,m,n}\left(\psi^m(\cdot)-\psi^m(x_0)\right);\,x_0\in Y_{l,m,n}\right\}$$
is a bounded subset of $\mathscr{D}(X_{l,m,n})$. From the continuity of the inclusion $\mathscr{D}(X_{l,m,n})\hookrightarrow\mathscr{D}\left(\overline{X}_{n+1}\right)$ $(X_{l,m,n}\subseteq\overline{X}_{n+1})$, we derive that for all $l,m\in\N, l>2/\delta_n$,
\begin{equation*}
B_{l,m,n}:=\left\{\frac{|\det J\psi^m(\cdot)|}{\prod_{j=0}^{m-1} w(\psi^j(\cdot))}\chi_{l,m,n}\left(\psi^m(\cdot)-\psi^m(x_0)\right);\,x_0\in Y_{l,m,n}\right\}
\end{equation*}
is a bounded subset of $\mathscr{D}(\overline{X}_{n+1})$.

Because $\mathscr{D}\left(\overline{X}_{n+1}\right)$ is metrizable, it follows from Mackey's countability condition (see e.g.\ \cite[Proposition 2.6.3]{Horvath} or \cite[Lemma 26.6 a)]{MeVo1997}) that there are a closed, absolutely convex, and bounded $B\subseteq\mathscr{D}\left(\overline{X}_{n+1}\right)$ and strictly positive numbers $\alpha_{l,m,n} (l,m\in\N, l>2/\delta_n)$
such that
\begin{equation}\label{inclusion 2}
B_{l,m,n}\subseteq\alpha_{l,m,n} B.	
\end{equation}

Let $B^\circ$ denote the polar of $B$ with respect to the dual pair $(\mathscr{D}(X),\mathscr{D}'(X))$. Now, as $\mathscr{G}(X)$ and $\mbox{proj}_{\leftarrow j}(\mathscr{G}\left(X_j),r_{X_{j+1}}^{X_j}\right)$ are topologically isomorphic, from the topologizability of $C_{w,\psi}$ it follows that for the zero neighborhood $B^\circ\cap\mathscr{G}(X)$ in $\mathscr{G}(X)$ there is $k\in\N$ and a zero neighborhood $U_k$ in $\mathscr{G}(X_k)$ such that for all $m\in\N$ there are $\gamma_m$ with
\begin{equation}\label{candidate for k}
C_{w,\psi}^m\left((r_X^{X_k})^{-1}(U_k)\right)\subseteq \gamma_m (B^\circ\cap\mathscr{G}(X))\subseteq\gamma_m B^\circ.
\end{equation}
We shall show that $\psi^m\left(\overline{X}_n\right)\subseteq\overline{X}_k$ for all $m\in\N$. Taking polars with respect to the dual pair $(\mathscr{D}(X),\mathscr{D}'(X))$, (\ref{candidate for k}) together with the Bipolar Theorem (cf.\ \cite[Theorem 22.13]{MeVo1997})) implies
$$\forall\,m\in\N:\,\left(C_{w,\psi}^t\right)^m(B)\subseteq\gamma_m\left(\left(r_X^{X_k}\right)^{-1}(U_k)\right)^\circ,$$
where $C_{w,\psi}^t$ denotes the transpose of $C_{w,\psi}$ on $\mathscr{D}(X)$. By (\ref{inclusion 2}) we deduce
\begin{equation}\label{inclusion 3}
\forall\,l,m\in\N, l>2/\delta_n: \left(C_{w,\psi}^t\right)^m(B_{l,m,n})\subseteq\alpha_{l,m,n}\gamma_m\left(\left(r_X^{X_k}\right)^{-1}(U_k)\right)^\circ.
\end{equation}
In order to show $\psi^m\left(\overline{X}_n\right)\subseteq\overline{X}_k, m\in\N,$ we argue by contradiction. We assume the existence of $m_0\in\N$ and
$$x_0\in\left\{x\in X_n;\,\prod_{j=0}^{m_0-1}w(\psi^j(x))\neq 0\right\}$$
such that $\psi^{m_0}(x_0)\notin \overline{X}_k$. Then there is $l\in\N, l>2/\delta_n,$ with $x_0\in Y_{l,m_0,n}$ and because $Y_{l+1,m_0,n}\subseteq Y_{l,m_0,n}$ we can have $l_0$ so large that $x_0\in Y_{l_0,m_0,n}$ and $\overline{B(\psi^{m_0}(x_0),1/l_0)}\subseteq X\backslash\overline{X}_k$ and such that according to hypothesis a) $r_X^{X_k\cup U}$ has dense range for $U:=B(\psi^{m_0}(x_0),1/l_0)$.

Choose $h\in\mathscr{G}\left(B(\psi^{m_0}(x_0),1/l_0)\right)$ for
$$\varphi:=\chi_{l_0,m_0,n}(\cdot -\psi^{m_0}(x_0))=\chi_{\varepsilon_{l_0,m_0,n}}(\cdot-\psi^{m_0}(x_0))$$
according to hypothesis b) where without loss of generality we assume that
$$\delta_\varphi(h):=\langle h,\varphi\rangle =1.$$
By the properties of a sheaf, there is $v\in\mathscr{G}(X_k\cup U)$ such that $r_{X_k\cup U}^{X_k}(v)=0$ and $r_{X_k\cup U}^U(v)=3\alpha_{l_0,m_0,n}\gamma_{m_0} h$. Since $\varphi\in\mathscr{D}(U)$ we have $\delta_\varphi(v)=3 \alpha_{l_0,m_0,n}\gamma_{m_0}$.

Because $r_X^{X_k\cup U}$ has dense range by hypothesis a) there is $u\in\mathscr{G}(X)$ such that
$$r_X^{X_k\cup U}(u)-v\in\left(r_{X_k\cup U}^{X_k}\right)^{-1}(U_k)\cap\delta_\varphi^{-1}\left(B(0,\alpha_{l_0,m_0,n}\gamma_{m_0})\right),$$
where $U_k$ is the zero neighborhood in $\mathscr{G}(X_k)$ from (\ref{candidate for k}), so that
\begin{equation}\label{preparing contradiction}
\delta_\varphi\left(r_X^{X_k\cup U}(u)\right)\in B(3\alpha_{l_0,m_0,n}\gamma_{m_0},\alpha_{l_0,m_0,n}\gamma_{m_0})
\end{equation}
as well as
$$r_X^{X_k}(u)=r_{X_k\cup U}^{X_k}\left(r_X^{X_k\cup U}(u)-v\right)+r_{X_k\cup U}^{X_k}(v)=r_{X_k\cup U}^{X_k}\left(r_X^{X_k\cup U}(u)-v\right)+0\in U_k,$$
that is
$$u\in \left(r_X^{X_k}\right)^{-1}(U_k).$$
Because by definition of $B_{l_0,m_0,n}$ and $x_0\in Y_{l_0,m_0,n}$ we have
$$\frac{|\det J\psi^{m_0}(\cdot)|}{\prod_{j=0}^{m_0-1} w(\psi^j(\cdot))}\chi_{l_0,m_0,n}\left(\psi^{m_0}(\cdot)-\psi^{m_0}(x_0)\right)\in B_{l_0,m_0,n}$$
it follows herefrom, (\ref{inclusion 3}), and (\ref{preparing contradiction})
\begin{align*}
\alpha_{l_0,m_0,n}\gamma_{m_0}\geq&\left|\left\langle \left(C_{w,\psi}^t\right)^{m_0}\left(\frac{|\det J\psi^{m_0}(\cdot)|}{\prod_{j=0}^{m_0-1} w(\psi^j(\cdot))}\chi_{l_0,m_0,n}\left(\psi^{m_0}(\cdot)-\psi^{m_0}(x_0)\right)\right),u\right\rangle\right|\\
=&\left|\left\langle \left[\frac{\prod_{j=0}^{m_0-1}w(\psi^j(\cdot))}{|\det J\psi^{m_0}(\cdot)|}\frac{|\det J\psi^{m_0}(\cdot)|}{\prod_{j=0}^{m_0-1} w(\psi^j(\cdot))}\chi_{l_0,m_0,n}\left(\psi^{m_0}(\cdot)-\psi^{m_0}(x_0)\right)\right]\right.\right.\\
&\left. \left. \circ(\psi^{m_0})^{-1},u\right\rangle\right|\\
=&\left|\left\langle\chi_{l_0,m_0,n}(\cdot -\psi^{m_0}(x_0)), u\right\rangle\right|=\left|\delta_\varphi\left(r_X^{X_k\cup U}(u)\right)\right|>2\alpha_{l_0,m_0,n}\gamma_{m_0}
\end{align*}
which gives a contradiction.

Therefore,
$$\forall\,m\in\N:\,\psi^m\left(\left\{x\in X_n;\,\prod_{l=0}^{m-1}w(\psi^l(x))\neq 0\right\}\right)\subseteq\overline{X}_k.$$
Because $\psi^m$ is continuous and because
$$\left\{x\in X_n;\,\prod_{l=0}^{m-1}w(\psi^l(x))\neq 0\right\}=\bigcap_{l=0}^{m-1}\left\{x\in X_n;\,w(\psi^l(x))\neq 0\right\}$$
is dense in $X_n$ we conclude
$$\forall\,m\in\N:\,\psi^m\left(\overline{X}_n\right)\subseteq\overline{X}_k.$$
Because $n$ was arbitrarily chosen and $(X_n)_{n\in\N}$ is an open, relatively compact exhaustion of $X$ it finally follows that $\psi$ has stable orbits.\hfill$\square$\\

Combining Proposition \ref{stable sufficient for distributions} and Lemma \ref{stable necessary distributions} we obtain a characterization of topologizability for weighted composition operators.

\begin{theorem}\label{characterization topologizability distributions}
	Let $X\subseteq\R^d$ be open, $w\in C^\infty(X)$, $\psi:X\rightarrow X$ smooth and injective such that $\det J\psi(x)\neq 0$ for all $x\in X$. Assume that $\mathscr{G}$ is a sheaf of distributions defined by local properties such that $C_{w,\psi}(\mathscr{G}(X))\subseteq\mathscr{G}(X)$ and that additionally the following conditions hold.
	\begin{enumerate}
		\item[a)] There is an open, relatively compact exhaustion $(X_n)_{n\in\N}$ of $X$ such that for each $n\in\N$ and every $x\in X\backslash\overline{X}_n$  and every $\varepsilon>0$ for which $\overline{B(x,\varepsilon)}\subseteq X\backslash{X}_n$ the restriction $r_X^{X_n\cup B(x,\varepsilon)}$ has dense range.
		\item[b)] There is $\varepsilon_0>0$ such that for all $\varepsilon\in (0,\varepsilon_0)$ there is $\chi_\varepsilon\in\mathscr{D}(B(0,\varepsilon))$ such that for all $x\in X$ with $\overline{B(x,\varepsilon)}\subseteq X$ there is $h\in\mathscr{G}(B(x,\varepsilon))$ satisfying $\langle h,\tau_x\chi_\varepsilon\rangle\neq 0$, where $\tau_x\chi_\varepsilon(y):=\chi_\varepsilon(y-x)$.
		\item[c)] For every $l\in\N_0$ the set $\left\{x\in X;\,w(\psi^l(x))\neq 0\right\}$ is dense in $X$.
	\end{enumerate}
	Then for the weighted composition operator $C_{w,\psi}$ on $\mathscr{G}(X)$ the following are equivalent.
	\begin{enumerate}
		\item[i)] $C_{w,\psi}$ is topologizable.
		\item[ii)] $\psi$ has stable orbits. 
	\end{enumerate}
\end{theorem}

By Remark \ref{sufficient a) and b) distributions}, the conditions a) and b) in the above theorem are satisfied for $\mathscr{G}=\mathscr{D}'$ and $\mathscr{G}=\mathscr{D}'_P$ for certain $P$ while condition c) is fulfilled whenever $\left\{x\in X;\, w(x)\neq 0\right\}$ is dense in $X$. In particular, we have the following.

\begin{corollary}\label{characterization topologizability all distributions}
	Let $X\subseteq\R^d$ be open, $w\in C^\infty(X)$, $\psi:X\rightarrow X$ smooth and injective such that $\det J\psi(x)\neq 0$ for all $x\in X$. Moreover, assume that $\left\{x\in X;\, w(x)\neq 0\right\}$ is dense in $X$. Then, the following are equivalent.
	\begin{enumerate}
		\item[i)] The weighted composition operator $C_{w,\psi}$ is topologizable on $\mathscr{D}'(X)$.
		\item[ii)] $\psi$ has stable orbits.
	\end{enumerate}
\end{corollary}

Now, we turn our attention to power boundedness. For a smooth and injective $\psi:X\rightarrow X$ with $\det J\psi(x)\neq 0$ for all $x\in X$ it follows that $\psi^m(X)$ is an open subset of $\R^d$ and $\psi^m:X\rightarrow\psi^m(X)$ is a diffeomorphism for every $m\in\N$. In particular, $\left(\psi^m\right)^{-1}:\psi^m(X)\rightarrow\R^d$ is a smooth function. Its components are denoted by $\left(\psi^m\right)^{-1}_c, 1\leq c\leq d$. For $Y\subseteq\R^d$ open, $K\subseteq Y$ compact, $n\in\N_0$, and $f\in C^\infty(Y)$ we define $\|f\|_{n,K}:=\sup_{|\alpha|\leq n, x\in K}|\partial^{\alpha} f(x)|$. Thus, $\|\cdot\|_{n,K}$ is a seminorm on $C^\infty(Y)$ and the standard topology on $C^\infty(Y)$ is the one generated by the set of seminorms $\left\{\|\cdot\|_{n,K};\,n\in\N_0, K\subseteq Y\mbox{ compact}\right\}$.

\begin{theorem}\label{some results on power boundedness}
	Let $X\subseteq\R^d$ be open, let $w\in C^\infty(X)$ be such that the set $\left\{x\in X;\,w(x)\neq 0\right\}$ is dense in $X$, and let $\psi:X\rightarrow X$ be smooth and injective with $\det J\psi(x)\neq 0$ for all $x\in X$. Then, among the following, i) implies ii) and ii) implies iii).
	\begin{enumerate}
		\item[i)] $\psi$ has stable orbits and for every compact set $K\subseteq X$ we have
		\begin{equation}\label{similar to bounded}
		\sup_{m\in\N}\left\|\left(\prod_{j=0}^{m-1}\frac{w(\psi^j(\cdot))}{|\det J\psi(\psi^j(\cdot))|}\right)\circ\left(\psi^m\right)^{-1}\right\|_{n,\psi^m(K)}<\infty,
		\end{equation}
		and
		\begin{equation}\label{is this also necessary}
		\forall\,1\leq c\leq d:\,\sup_{m\in\N}\left\|\left(\psi^m\right)_c^{-1}\right\|_{n,\psi^m(K)}<\infty.
		\end{equation}
		\item[ii)] $C_{w,\psi}$ is power bounded on $\mathscr{D}'(X)$.
		\item[iii)] $\psi$ has stable orbits and for every compact set $K\subseteq X$ (\ref{similar to bounded}) holds.
	\end{enumerate}
\end{theorem}

\begin{proof}
	Assume that i) is valid. For a compact subset $K\subseteq X$ let $L(K)$ be a compact subset of $X$ satisfying $\psi^m(K)\subseteq L(K)$ for all $m\in\N$. Let $B\subseteq\mathscr{D}(X)$ be bounded and let $K\subseteq X$ be compact such that $B\subseteq\mathscr{D}(K)$ is bounded. 
	
	For fixed $u\in\mathscr{D}'(X)$ there are $r\in\N_0$ and $M_1>0$ such that
	$$\forall\,\phi\in\mathscr{D}(L(K)):\,|\langle u,\phi\rangle|\leq M_1\|\phi\|_{r,L(K)}.$$
	Moreover, as $B\subseteq\mathscr{D}(K)$ is bounded, there is $M_2>0$ with
	$$\forall\,\varphi\in B:\,\|\varphi\|_{r,K}\leq M_2.$$
	Because i) holds, there is $M_3>0$ satisfying
	$$\sup_{m\in\N}\left\|\left(\prod_{j=0}^{m-1}\frac{w(\psi^j(\cdot))}{|\det J\psi(\psi^j(\cdot))|}\right)\circ(\psi^m)^{-1}\right\|_{n,\psi^m(K)}<M_3.$$
	Since for smooth functions $f$ and $g$
	$$\|f g\|_{r,\psi^m(K)}\leq 2^r\|f\|_{r,\psi^m(K)} \|g\|_{r,\psi^m(K)}$$
	it follows for $\varphi\in B\subseteq\mathscr{D}(B)$ with \cite[Proposition 3.10 ii)]{Kalmes19-2} applied to $(\psi^m)^{-1}$ in place of $\psi$ and $m=1$ in the context of the cited proposition
	\allowdisplaybreaks 
	\begin{align*}
	|\langle C_{w,\psi}^m(u),\varphi\rangle|&=\left|\left\langle u, \left(\varphi\prod_{j=0}^{m-1}\frac{w(\psi^j(\cdot))}{|\det J\psi(\psi^j(\cdot))|}\right)\circ\left(\psi^m\right)^{-1}\right\rangle\right|\\
	&\leq M_1\left\|\left(\varphi\prod_{j=0}^{m-1}\frac{w(\psi^j(\cdot))}{|\det J\psi(\psi^j(\cdot))|}\right)\circ\left(\psi^m\right)^{-1}\right\|_{r,L(K)}\\
	&=M_1\left\|\left(\varphi\prod_{j=0}^{m-1}\frac{w(\psi^j(\cdot))}{|\det J\psi(\psi^j(\cdot))|}\right)\circ\left(\psi^m\right)^{-1}\right\|_{r,\psi^m(K)}\\
	&=M_1\left\|\left(\prod_{j=0}^{m-1}\frac{w(\psi^j(\cdot))}{|\det J\psi(\psi^j(\cdot))|}\circ\left(\psi^m\right)^{-1}\right)\;\varphi\circ\left(\psi^m\right)^{-1}\right\|_{r,\psi^m(K)}\\
	&\leq 2^r M_1\left\|\prod_{j=0}^{m-1}\frac{w(\psi^j(\cdot))}{|\det J\psi(\psi^j(\cdot))|}\circ\left(\psi^m\right)^{-1}\right\|_{r,\psi^m(K)}\times\\
	&\;\;\;\;\;\;\;\;\;\times\left\|\varphi\circ\left(\psi^m\right)^{-1}\right\|_{r,\psi^m(K)}\\
	&\leq 2^r M_1 M_3 M_4\;\left\|\varphi\right\|_{r,K} \left(1+\max_{1\leq c\leq d}\left\|\left(\psi^m\right)_c^{-1}\right\|_{r,\psi^m(K)}\right)^r\\
	&\leq 2^r M_1 M_3 M_4\, M_2 \left(1+\max_{1\leq c\leq d}\left\|\left(\psi^m\right)_c^{-1}\right\|_{n,\psi^m(K)}\right)^r
	\end{align*}
	for a suitable constant $M_4$ which is independent of $\varphi$. Because of i) we thus obtain
	\begin{equation}\label{pointwise bounded}
	\forall\,u\in\mathscr{D}'(X):\,\sup_{m\in\N}\sup\left\{\left|\left\langle C_{w,\psi}^m(u),\varphi\right\rangle\right|;\,\varphi\in B\right\}<\infty.
	\end{equation}
	As the strong dual of the complete Schwartz space $\mathscr{D}(X)$, $\mathscr{D}'(X)$ is ultrabornological (see e.g.\ \cite[Proposition 24.23]{MeVo1997}), hence barrelled so that (\ref{pointwise bounded}) implies ii).
	
	Assume that ii) holds. Then $C_{w,\psi}$ is topologizable and from Remark \ref{sufficient a) and b) distributions} and Theorem \ref{characterization topologizability distributions} it follows that $\psi$ has stable orbits. As before, for $K\subseteq X$ compact we denote by $L(K)$ a compact subset of $X$ for which $\psi^m(K)\subseteq L(K)$ holds for every $m\in\N$.
	
	We choose $\varphi\in\mathscr{D}(X)$ with $\varphi=1$ in a neighborhood of $K$. Then
	\begin{align*}
	\forall\,m\in\N:\;\left(\varphi \prod_{j=0}^{m-1}\frac{w(\psi^j(\cdot))}{|\det J\psi(\psi^j(\cdot))|}\right)\circ\left(\psi^m\right)^{-1}&\in\mathscr{D}\left(\psi^m(\supp\varphi)\right)\\
	&\subseteq\mathscr{D}\left(L(\supp\varphi)\right)
	\end{align*}
	and because for arbitrary $u\in\mathscr{D}'(X)$
	$$\forall\, m\in\N:\,\left\langle C_{w,\psi}^m(u),\varphi\right\rangle=\left\langle u,\left(\varphi\prod_{j=0}^{m-1}\frac{w(\psi^j(\cdot))}{|\det J\psi(\psi^j(\cdot))|}\right)\circ\left(\psi^m\right)^{-1}\right\rangle$$
	it follows from the power boundedness of $C_{w,\psi}$ that
	$$\left\{\left(\varphi\prod_{j=0}^{m-1}\frac{w(\psi^j(\cdot))}{|\det J\psi(\psi^j(\cdot))|}\right)\circ\left(\psi^m\right)^{-1};\,m\in\N\right\}$$
	is weakly bounded in $\mathscr{D}(X)$ and therefore, by Mackey's Theorem (see e.g.\ \cite[Theorem 23.15]{MeVo1997}), bounded in $\mathscr{D}(X)$. By the choice of $\varphi$ we have
	\begin{align*}
	&\left(\varphi\prod_{j=0}^{m-1}\frac{w(\psi^j(\cdot))}{|\det J\psi(\psi^j(\cdot))|}\right)\circ\left(\psi^m\right)^{-1}|_{\psi^m(K)}\\
	=&\left(\prod_{j=0}^{m-1}\frac{w(\psi^j(\cdot))}{|\det J\psi(\psi^j(\cdot))|}\right)\circ\left(\psi^m\right)^{-1}|_{\psi^m(K)}
	\end{align*}
	for all $m\in\N$ so that (\ref{similar to bounded}) follows. Thus, ii) implies iii).
\end{proof}

\begin{remark}\label{power boundedness on test functions}
	For a diffeomorphism $\psi:X\rightarrow X$ it is straightforward to calculate that the transpose of $C_{w,\psi}^m$ on $\mathscr{D}'(X)$ is given by the restriction of $C_{w_\psi,\psi^{-1}}^m$ to $\mathscr{D}(X)$, where
	$$w_\psi:X\rightarrow\C, w_\psi(x)=\frac{w}{|\det J\psi|}\circ\psi^{-1}.$$
	Since $\mathscr{D}'(X)$ is the strong dual of the complete Schwartz space $\mathscr{D}(X)$ it follows that $\mathscr{D}'(X)$ is ultrabornological (see e.g.\ \cite[Proposition 24.23]{MeVo1997}), hence barrelled.  Thus, $C_{w,\psi}$ is power bounded if and only if $\left\{C_{w,\psi}^m(u);\,m\in\N_0\right\}$ is bounded in $\mathscr{D}'(X)$ for every $u\in\mathscr{D}'(X)$. Because $\mathscr{D}(X)$ is reflexive, by Mackey's Theorem (\cite[Theorem 23.15]{MeVo1997}) it follows that the latter is equivalent to the boundedness of $\left\{\left\langle C_{w,\psi}^m(u),\varphi\right\rangle,\,m\in\N_0\right\}$ for all $u\in\mathscr{D}'(X), \varphi\in\mathscr{D}(X)$. Applying Mackey's Theorem once more, this in turn is equivalent to $\left\{C_{w_\psi,\psi^{-1}}^m(\varphi);\,m\in\N_0\right\}$ being bounded in $\mathscr{D}(X)$. From the barrelledness of $\mathscr{D}(X)$ it finally follows that this is equivalent to $C_{w_\psi,\psi^{-1}}$ being power bounded on $\mathscr{D}(X)$.
	
	As usual, for bijective $\psi$ we write $\psi^{-m}$ instead of $(\psi^{-1})^m, m\in\N$:
\end{remark}

\begin{corollary}
	Let $X\subseteq\R^d$ be open and $\psi:X\rightarrow X$ be smooth and bijective. Then the following are equivalent:
	\begin{enumerate}
		\item[i)] The composition operator $C_\psi$ is power bounded on $\mathscr{D}'(X)$.
		\item[ii)] $\psi$ has stable orbits and for every compact subset $K\subseteq X$ we have
		\begin{equation*}
		\forall\,1\leq c\leq d:\,\sup_{m\in\N}\left\|\left(\psi^{-m}\right)_c\right\|_{n,\psi^m(K)}<\infty.
		\end{equation*}
	\end{enumerate}
\end{corollary}

\begin{proof}
	Assuming that $C_\psi$ is power bounded on $\mathscr{D}'(X)$ it follows from Theorem \ref{some results on power boundedness} that $\psi$ has stable orbits. Thus, given $K\subseteq X$ compact we can choose $L\subseteq X$ compact such that $\psi^m(K)\subseteq L$ for every $m\in\N$. Additionally, we choose $\phi\in\mathscr{D}(X)$ with $\phi=1$ in a neighborhood of $L$. For $1\leq c\leq d$ we define $\varphi_c(x)=x_c\phi(x)$ so that $\varphi_c\in\mathscr{D}(X)$. Since $C_\psi$ is power bounded on $\mathscr{D}'(X)$ it follows from Remark \ref{power boundedness on test functions} together with $w_\psi=1$ that $\left\{C^m_{\psi^{-1}}(\varphi_c);\,m\in\N\right\}$ is a bounded subset of $\mathscr{D}(X)$. In particular, taking into account that $C_{\psi^{-1}}^m(\varphi_c)=\varphi_c\circ\psi^{-m}=(\psi^{-m})_c$ in a neighborhood of $K$, we obtain
	\begin{eqnarray*}
		\forall\,n\in\N:\,\infty&>&\sup_{m\in\N}\left\|C_{\psi^{-1}}^m(\varphi_c)\right\|_{n,L}\geq\sup_{m\in\N}\left\|C_{\psi^{-1}}^m(\varphi_c)\right\|_{n,\psi^m(K)}\\
		&= &\sup_{m\in\N}\left\|(\psi^{-m})_c\right\|_{n,\psi^{m}(K)},
	\end{eqnarray*}
	so that ii) follows.
	
	If on the other hand ii) holds, it follows from
	\[\det J\psi^{-m}=\det\left(\prod_{j=1}^m\left(J\psi\right)^{-1}(\psi^{-j}(\cdot))\right)=\left(\prod_{j=0}^{m-1}\frac{1}{\det J\psi(\psi^j(\cdot))}\right)\circ(\psi^m)^{-1}\]
	and the fact that for fixed $m\in\N$ and for every multi-index $\alpha\in\N_0^d$ the function $\partial^\alpha\det J\psi^{-m}$ is a polynomial in $\partial^\beta\left(\psi^{-m}\right)_c, 1\leq c\leq d, 1\leq|\beta|\leq |\alpha|+1,$ with integer coefficients independent of $m$ that for arbitrary compact subset $K\subseteq X$
	\[\sup_{m\in\N}\left\|\left(\prod_{j=0}^{m-1}\frac{w(\psi^j(\cdot))}{|\det J\psi(\psi^j(\cdot))|}\right)\circ(\psi^m)^{-1}\right\|_{n,\psi^m(K)}<\infty.\]
	Thus, ii) implies condition i) of Theorem \ref{some results on power boundedness} and thus the power boundedness of $C_\psi$ on $\mathscr{D}'(X)$. 
\end{proof}

\bibliographystyle{plain}
\bibliography{../../../../bib_Thomas}

\end{document}